\documentclass[11pt]{amsart}

\usepackage{amsfonts}
\usepackage{amsmath}
\usepackage{calligra}
\usepackage{amssymb}
\usepackage{ytableau}
\usepackage{mathabx}
\usepackage{graphicx}

\RequirePackage{color}[1999/02/16 v1.0i Standard LaTeX Color (DPC)]

\newif\ifcomments

\let\newComments\newKibitzer

\newComments\JV{JV}{red}
\newComments\EG{EG}{blue}

\newcommand{\e}{\mathfrak{e}}

\newcommand{\g}{\mathfrak{g}}
\newcommand{\h}{\mathfrak{h}}
\renewcommand{\k}{{\mathfrak k}}
\renewcommand{\l}{\mathfrak l}
\newcommand{\p}{\mathfrak p}
\newcommand{\q}{\mathfrak q}

\newcommand{\z}{Z_\h}
\renewcommand{\t}{{\mathfrak t}}
\renewcommand{\u}{{\mathfrak u}}

\newcommand{\C}{{\mathbb C}}

\newcommand{\R}{{\mathbb R}}
\newcommand{\Z}{{\mathbb Z}}

\newcommand{\res}{{\rm{res}}}
\newcommand{\supp}{{\rm{supp}}}
\newcommand{\Ad}{{\rm{Ad}}}
\newcommand{\diag}{{\rm{diag}}}
\newcommand{\Hom}{{\rm{Hom}}}

\newtheorem{thm}{Theorem}[section]
\newtheorem{prop}[thm]{Proposition}
\newtheorem{lem}[thm]{Lemma}
\newtheorem{cor}[thm]{Corollary}
\theoremstyle{definition}

\theoremstyle{remark}
\newtheorem{rmk}[thm]{Remark}
\newtheorem{fact}[thm]{Fact}

\numberwithin{equation}{section}

\begin{document}

\title{Discrete series of simple Lie groups with admissible restriction to a $SL_2(\mathbb R)$}
\author{Esther Galina,   Jorge A.  Vargas}
\thanks {Partially supported by  CONICET, SECYT-UNC (Argentina) }
\date{\today }
\keywords{ Discrete Series, branching laws, admissible restriction. }
\subjclass[2010]{Primary 22E46; Secondary 17B10}
\address{ FAMAF-CIEM, Ciudad Universitaria, 5000 C\'ordoba, Argentine}
\email{galina@famaf.unc.edu.ar, vargas@famaf.unc.edu.ar}

 \begin{abstract}
In this note we determine the irreducible square integrable representations of a simple group which admits an admissible restriction to a subgroup $H$ locally isomorphic to $SL_2(\mathbb R).$ We show such representation is holomorphic and we determine the essentially unique  $H$ with this property as well as multiplicity formulae.
  \end{abstract}

\maketitle
\markboth{Galina- Vargas}{Branching to   $SL_2(\mathbb R)$}

\section{introduction}
Let $G$ be a connected simple matrix Lie group and $K$ a fixed maximal compact subgroup of $G$. We assume that both groups $K$ and $G$  has the same rank.  From now on  $T \subset K$ is a fixed maximal torus. Therefore, $T$ is a compact Cartan subgroup of $G$. Under these hypothesis, Harish-Chandra showed there exists irreducible unitary representations of $G$ so that its matrix coefficients are square integrable with respect to a Haar measure on $G$. Let $H$ denote   the image of a nontrivial continuous morphism from $SL_2(\mathbb R)$ into $G$. The aim of this note is to determine the pairs $((\pi, V),H)$ of irreducible square integrable representations $(\pi, V)$ of $G$ which admits admissible restriction to $H$.  That is, $\res_H  \pi$ is an admissible representation, or equivalently, $res_H \pi$   is equal to a discrete Hilbert sum of irreducible representations of $H$ and the multiplicity of each irreducible factor is finite.

To state the main results we need to fix some more notation. The Lie algebra of a Lie group is denoted by the corresponding lower case Fraktur font and the complexification of a real Lie algebra, or a vector space, is denoted by adding the subscript $\C$. The conjugation of $X \in  \g_\C$ with respect to the real form $\g$ of $\g_\C$ is denoted by $\bar X$.
Denote by $\Phi (\g,\t)$ (resp. $\Phi(\k,\t)$) the root system of $\t_\mathbb C$ in $\g_\mathbb C$ (resp. $\t_\mathbb C$ in $\k_\mathbb C$).  The set of noncompact roots is defined by  $\Phi_n = \Phi (\g,\t) \smallsetminus \Phi(\k,\t).$     We write $\g_\mathbb C= \k_\mathbb C \oplus \p_\mathbb C$ the Cartan decomposition corresponding to the pair $(\g, K)$. Then, the root vectors associated to $\alpha \in \Phi_n$ (resp. $\alpha \in \Phi(\k,\t)$) lie in $\p_\mathbb C$ (resp. in $\k_\mathbb C$). Recall that a system of positive roots $\Psi $ of the root system $\Phi(\g, \t)$ is {\it holomorphic} whenever the sum $\alpha +\beta$   is not a root for every pair $\alpha, \beta \in \Psi_n:=\Psi \cap \Phi_n$. For a unitary representation $(\pi, V), $ the action of $X \in \g_\C$ on the subspace of $K$-finite vectors is denoted by $\dot \pi (X).$  A unitary irreducible representation $(\pi, V)$ is {\it holomorphic} when there exists a holomorphic system of positive roots $\Psi \subset \Phi(\g,\t)$ and a nonzero $K$-finite vector $ v \in V$ so that $\dot{\pi}(X_{-\gamma}) v=0$ for every noncompact root $\gamma$ in $\Psi$ and for every root element $X_\gamma$ associated to $\gamma$.

After the work of Harish-Chandra and other researchers it has been shown that holomorphic system of positive roots, as well as holomorphic unitary irreducible representations, do exist if and only if the Riemannian manifold $G/K$ admits a $G$-invariant complex structure. That is, $G/K$ is a Hermitian symmetric space.

We are ready to state our first result.

\begin{thm} \label{prop:res_H}
Assume there exists a square integrable irreducible representation $\pi$  for $G$ and  a subgroup $H$ locally isomorphic to $SL_2(\mathbb R)$ such that the restriction of $\pi$ to $H$ is an admissible representation. Then, $G/K$ is an Hermitian symmetric space and $\pi$ is a holomorphic representation of $G$.
\end{thm}

In order to list the pairs $((\pi, V),H) $ we are searching for, we recall on one side the list of pairs of Hermitian symmetric Lie algebras $(\g, \k)$ and a classification, up to conjugacy, of the subgroups of $G$ locally isomorphic to $SL_2(\mathbb R)$.
The list of Hermitian symmetric pairs is:

\smallskip

$ AIII \,\, (\mathfrak{su}(p,q), \mathfrak{su}(p) \oplus \mathfrak{su}(q)+\mathfrak{so}(2))$

$BDI \,\, (\mathfrak{so}(p,2), \mathfrak{so}(p) \oplus\mathfrak{so}(2))$

$CI \,\, (\mathfrak{sp}(n,\mathbb R), \mathfrak{u}(n))$

$ DIII \,\, (\mathfrak{so}^\star(2p), \u(p))$

$ EIII \,\, (\e_{6(-14)}, \mathfrak{so}(5)\oplus\mathfrak{so}(2))$

$  EVII  \,\, (\e_{7(-25)}, \e_6 \oplus \mathfrak{so}(2))$

Among them, only $AIII (p=q$), $BDI, CI, DIII$ ($p$ even) and $EVIII$ corresponds to  tube domains.

\smallskip

The classification of the Lie subalgebras $\h$  of $\g$ isomorphic to $\mathfrak{sl}_2(\mathbb R)$, up to conjugacy,  has been obtained by Kostant-Rallis. In particular, they show that the number of conjugacy classes is finite.  An explicit list of representatives of the conjugacy classes for group $G$ has been obtained by
 \cite{Oh}, \cite{Dk2} for classical real Lie groups and in  \cite{Dk1} for the exceptional groups. In order to state the version of the   classification more suitable for this work,  to the end of this introduction we assume \begin{center} $G/K$ is a Hermitian symmetric space. \end{center} We fix a holomorphic system of positive roots $\Psi \subset \Phi(\g, \t)$, and   we denote the simple roots of $\Psi$ by
\begin{equation} \label{eqn:simple roots}
\{\beta_1, \dots, \beta_\ell\} \qquad \text{where}\, \beta_1, \dots, \beta_{\ell-1} \in \Phi(\k,\t)\text{ and }  \beta_\ell \in \Phi_n.
\end{equation}
The highest root is written as $\beta_M = \sum_j c_j \beta_j $ with $c_j \geq 1 $ for all $j$ and $c_\ell =1$.

Under this setting, they  have shown that a subalgebra of $\g$ isomorphic to $\mathfrak{sl}_2(\mathbb R) $ is conjugated by an element of $G$ to a subalgebra $\h$ whose complexification is  spanned by a normal triple, or more precisely, a KS-triple
$$
Z_\h \in i(\t \cap \h),\quad E_\h \in \p_\C \cap \h_\C, \quad F_\h \in \p_\C \cap \h_\C
$$
which satisfies,
\begin{equation}
\begin{aligned}\label{eqn:KS-triple}
[Z_\h,E_\h]= 2E_\h,& \quad [Z_\h,F_\h]=-2F_\h,\quad [E_\h,F_\h]=Z_\h, \quad \bar{E_\h}=F_\h\\
\alpha (Z_\h) &\geq 0 \qquad \text{for every}\, \alpha \in \Psi_c:=\Psi \cap \Phi(\k, \t)
\end{aligned}
\end{equation}
Moreover, the characteristic vector  $Z_\h$ of the $KS$-triple determines the $K_\C$-conjugacy class of $\h$.  That is, if $ Z_\h = Z_{\h'}$ then $\h_\C$ is $K_\C$-conjugated to $\h'$ (see Theorem 9.4.4 of \cite{CM}).

The finite set of vectors $\z$ which parametrize the $G$-conjugacy classes of $SL_2$-subgroups, can be expressed explicitly in terms of the dual basis to the basis of simple roots \eqref{eqn:simple roots}.  The coordinates of each $Z_\h$  has been computed for the classical Lie algebras in \cite{Dk2}, \cite{Oh}  and for the exceptional ones in \cite{Dk1}. Moreover, for the classical Lie algebras, they compute the eigenvalues of the matrix $\z$.  More precisely, if $Z_j, j=1, \dots, \ell$, denote the dual basis to \eqref{eqn:simple roots}. That is, the elements of $\t_\C$,  $Z_j$ are  such that $\beta_r (Z_s)=\delta_{rs}$, hence, we have
$$
Z_\h =\sum_{1 \leq s \leq \ell} \beta_s(Z_\h) Z_s
$$
It follows  from \eqref{eqn:KS-triple} and the work of Kostant-Rallis that  $\beta_s(Z_\h)$ is a  non negative integer    for $s=1,\dots, \ell -1$  and that $\beta_\ell(Z_\h)$ is an integer.  The precise values of the integers $\beta_s(\z)$ for the exceptional Lie algebras are listed in \cite{Dk1}. In Proposition \ref{prop:Z=Z0} we show that there is a unique characteristic vector $Z_\h$ with $\beta_\ell (Z_\h)$ a positive number. We also show that the hypothesis that $(\pi, V)$ is a holomorphic square integrable representation with admissible restriction to $H$ forces either that $\beta_\ell (Z_\h) > 0$, or $\beta_\ell (Z_\h) < 0$ and $\beta_i(\z)=0$ for all $1\leq i \leq \ell-1$.
Next, we produce an explicit example of a KS-triple $\{Z_0, E_0, F_0\}$
such that any holomorphic representation  $\pi$ restricted to the corresponding $SL_2$-subgroup  $H_0$ is admissible. Later on we show that $H_0$ is unique up to conjugacy by elements of $G$.

We denote by  $(\, ,\,)$ the inner product on $i\t^*$ associated to the Killing form on $\g$.  From the holomorphic system $\Psi$, Harish-Chandra has constructed a strongly orthogonal spanning  set
\begin{equation} \label{eqn:S}
S=\{ \gamma_1, \dots, \gamma_r \}  \subset \Psi_n
\end{equation}
as follows: $\gamma_1 =\beta_M $ is the maximal root in $\Psi_n=\Psi \cap \Phi_n$,
  $\gamma_2$ is the maximal root in $\{ \gamma \in \Psi_n : \gamma \perp \gamma_1\}$ and so on. Here $r$ is the real rank of the group $G$.
Notice that there is no non compact root orthogonal to $S$. For each root $\gamma  \in \Psi_n$ we fix a KS-triple $Z_\gamma \in i\t$, $E_\gamma \in \g_\gamma \cap \p_\C$ and $ F_\gamma \in \g_\gamma \cap \p_\C$.
In particular, if $H_\gamma$ denotes the coroot vector of $\gamma$, then $Z_\gamma =\frac{2 H_\gamma}{(\gamma, \gamma)}$, $\bar E_\gamma = F_\gamma$ and  $[E_\gamma, F_\gamma ]= Z_\gamma$. Let
\begin{equation} \label{def:Z0}
Z_0 =Z_{\gamma_1}+\dots +Z_{\gamma_r}, \quad E_0=E_{\gamma_1}+\dots +E_{\gamma_r}, \quad F_0 =F_{\gamma_1}+\dots +F_{\gamma_r}.
\end{equation}
Since $S$ is a strongly orthogonal set of non compact roots it readily follows that
$\{Z_0, E_0, F_0\}$ is a KS-triple. Let $H_0$ the image of $SL_2(\mathbb R)$ in $G$ associated to this triple. Later on, in Lemma \ref{lem:Z0}, we show that $\beta_\ell (Z_0) >0.$ This will verify the first statement in the following theorem.

\

 \begin{thm} \label{prop:H=H0}
 \

 \begin{enumerate}
 \item[(i)]  Any holomorphic irreducible, square integrable representation of $G$ has an admissible restriction to $H_0$.
 \item[(ii)] Whenever a holomorphic irreducible, square integrable representation of $G$ has admissible restriction to  a  subgroup $H$ locally isomorphic to  $Sl_2(\mathbb R)$, then $H$ is conjugated under $G$ to $H_0.$
\end{enumerate}
 \end{thm}

The first statement follows from work of  Vergne, Jacobsen, T. Kobayashi, Mollers and Oshima, for a reference \cite{om}. However, we give an independent proof. In section 5 we compute the multiplicity and Harish-Chandra parameter of each irreducible $H_0$-factor.

This note is organized as follows,  in the introduction we state our main results but the one on multiplicity and Harish-Chandra parameter of each irreducible factor, which is presented in Section 5. In section 2 and 3  we respectively present the proof of the two theorems stated in the introduction. In section 4 we write down the groups involved in the statement of the results.

\section{Proof of Theorem \ref{prop:res_H}}

For each noncompact root $\beta$, the subalgebra $\mathfrak{sl}_2(\beta)$ of $\g$ spanned by the root elements associated to $\pm \beta$ is isomorphic to $\mathfrak{sl}(2,\mathbb C)$ and invariant under the conjugation of $\g_\mathbb C$ with respect to $\g.$ Thus, the real points of the subalgebra $\mathfrak{sl}_2(\beta)$ is an algebra isomorphic to $\mathfrak{sl}_2(\mathbb R).$  Let $H_\beta$ denote the analytic subgroup corresponding to this real subalgebra.

Let $H$ be as in Theorem \ref{prop:res_H} and $T$ a compact Cartan subgroup of $G$ as in the introduction. We may assume, after conjugation by an element of $G$, that $H\cap T=H\cap K$ is a maximal compact subgroup of $H.$
The hypothesis that $(\pi, V)$ restricted to $H$ is admissible and $\pi$ is square integrable,  imply that $\pi$ restricted to $H\cap T$ is an admissible representation, see \cite{DV}.
Since $H_\beta$ is invariant under conjugation by elements of $T$, $TH_\beta$ is an analytic subgroup of $G$ with $T$ as a maximal compact subgroup. To continue we formulate a useful result of Kobayashi.

\begin{fact}\label{fact:Kb1}
 For a given a closed subgroup $L$ of $G$ with finitely many connected component and invariant under the Cartan involution, we know that a maximal compact subgroup of $L$ is $L\cap K$. In this context, if $(\pi,V)$ is a irreducible representation of $G$ and the restriction of $\pi$ to $L\cap K$ is admissible, Theorem 1.2 of  \cite{Kb1}, guarantees that the restriction of $\pi$ to $L$ is admissible.
\end{fact}

Owing to the Fact \ref{fact:Kb1} we have that $\pi$ restricted to $TH_\beta$ is an admissible representation. Moreover, Kobayashi  shows that the subspace of $T$-finite vectors in $V$ is identical to the subspace of $K$-finite vectors (see \cite{Kb2}, Proposition 1.6). The hypothesis  $(\pi,V)$ is a square integrable representation,  implies that each of the $TH_\beta$- irreducible constituents is a square integrable representation of $TH_\beta$, finally,  the Harish-Chandra module underlying an  irreducible square integrable representations of $SL_2(\mathbb R)$ is a  Verma modules. Therefore, there exists a non zero $K$-finite vector $v \in V$ so that either  $\dot{\pi}(X_{-\beta})(v)=0$ or $\dot{\pi}(X_{\beta})(v)=0$.
 From Lemma 3.4 of \cite{Vo} we have that $\dot{\pi}(X_{\pm \gamma})$, with $\gamma \in  \Psi_n$, is injective for every nilpotent element unless  $G/K$ is a Hermitian symmetric space. In turn, it follows that  $(\pi, V)$ is either  a holomorphic representation or antiholomorphic  representation,  details can be found in \cite{Va}. Thus, the proof of Theorem \ref{prop:res_H} is finished.

\section{Proof of Theorem \ref{prop:H=H0}}

Let $\p^+ $    denote the subspace spanned by the root elements corresponding to the roots in $\Psi_n$. Then,   $\p^+$ is invariant under the action of $K_\C$ via the adjoint representation of $G_\C.$ Hence, the symmetric algebra $S(\p^+)$ is a $K$-module. Let $W$ denote the lowest $K$-type of the holomorphic discrete series representation $(\pi, V)$, then the space of $K$-finite vectors of $\pi$ is $K$-isomorphic to the $K$-module $S(\p^+)\otimes W$.

\begin{fact} \label{fact:res_L}
In \cite{Kb1}, \cite{DV}, it is pointed out that the restriction of $\pi$ to $L$, with $L$ a closed connected subgroup  of $K$,  is admissible if and only if the algebra of invariants $S(\p^+)^L$ consists of the set of constant elements.
\end{fact}

We apply  this criterium to understand admissible restriction to the subgroup $H\cap T =\exp(\mathbb R i Z_\h)$ where $H$ is locally isomorphic to $SL_2(\mathbb R)$ as in Theorem \ref{prop:res_H} and \eqref{eqn:KS-triple}.

\begin{prop} \label{prop:res to H cap T}
A holomorphic discrete series representation $(\pi, V)$ of $G$  restricted to the subgroup
$H\cap T=\exp(\mathbb R iZ_\h)$ is admissible if and only if one of the two following condition holds,
\begin{enumerate}
\item[(i)] $\beta_\ell (Z_\h)$ is positive, or
\item[(ii)] $\beta_\ell(Z_\h) $ is negative and $ \beta_j(Z_\h)=0$ for all $1\leq  j \leq \ell-1.$
\end{enumerate}
\end{prop}

\begin{proof}
 Indeed, the weights of $\t_\mathbb C$ in $S(\p^+)$ are     $\sum_{\gamma \in \Psi_n} c_\gamma \gamma$ with $c_\gamma$ a nonnegative integer for every $\gamma.$ Any root $\gamma \in \Psi_n$ is equal to $\beta_\ell $ plus a linear combination with non negative coefficients of the roots $\beta_1, \dots, \beta_{\ell-1}.$ Thus, owing to  \eqref{eqn:simple roots},  we have the inequalities
$$
\beta_\ell(Z_\h) \leq \gamma (Z_\h) \leq \beta_M (Z_\h), \text{for}\,  \gamma  \in \Psi_n.
$$
Hence,
$$
(\sum_{\gamma \in \Psi_n} c_\gamma) \beta_\ell (Z_\h) \leq \sum_{\gamma \in \Psi_n} c_\gamma \gamma (Z_\h) \leq (\sum_{\gamma \in \Psi_n} c_\gamma) \beta_M (Z_\h)
$$
Thus, according to the Fact \ref{fact:res_L}, the restriction of $\pi$ to $\exp(\mathbb R iZ_\h)$ is admissible if and only if
$$
\sum_{\gamma \in \Psi_n} c_\gamma \gamma (Z_\h) \not= 0 \qquad \text{for\, every } \sum_{\gamma \in \Psi_n} c_\gamma  > 0
$$
Hence, the converse implication in the proposition  is true.

For the direct implication,  if $\beta_\ell (Z_\h) =  0 $ then the  vectors $E_{\beta_\ell}^k \in S(\p^+)$, $k=1,2,\dots$,  span an infinite dimensional subspace of $\t\cap \h$-weight zero, which contradicts the admissible restriction of $\pi$  to
$T\cap H$. Thus, $\beta_\ell (Z_\h) \neq 0$. If $\beta_\ell (Z_\h) <  0 $ there are two possibilities: either  $\beta_k (Z_\h) >0$ for some $1 \leq  k \leq \ell-1$,  or $\beta_k (Z_\h)=0$ for all
$1 \leq k \leq \ell-1$. In the first situation we recall that $\beta_j (Z_\h) $ is a nonnegative integer for any $j \leq \ell-1$  and hence the space of vectors of $T$-weights
$\nu= n(-\beta_\ell (Z_\h) \beta_k + \beta_k (Z_\h) \beta_\ell)$, $n=0,1,2,\dots$, is infinite dimensional and consists of vectors of $\t\cap \h$-weight zero. This contradicts that $\pi$ is an admissible representation of $\t\cap \h$ and implies that the second situation holds.
\end{proof}

We observe that whenever  the situation (ii) holds,      there is another KS-triple which satisfies \eqref{eqn:KS-triple} for $\h$, namely, the triple $\{-Z_\h, F_\h, E_\h\}$. Now, $\beta_\ell (-Z_\h) >0$ and $\beta_j(-Z_\h)=0$,  $j=1, \dots \ell-1$, hence we are in situation (i).
Note that that both KS-triples generate the real Lie algebra $\h$.

For the particular element $Z_0$ defined in \eqref{def:Z0}, we obtain the following result in relation with Proposition \ref{prop:res to H cap T}.

\begin{lem} \label{lem:Z0}
Let $Z_0$ be as in \eqref{def:Z0}, then for every noncompact root $\beta$ in $\Psi$,  $\beta (Z_0) > 0$. In particular, for the simple noncompact root, $\beta_\ell (Z_0) > 0$.
\end{lem}

\begin{proof}
Since the system $\Psi$ is holomorphic, $\beta + \gamma_j$ never is a root. Hence, for every $j$, we have  $\frac{2(\beta, \gamma_j)}{(\gamma_j, \gamma_j)} \geq 0.$ By construction, $S$ is a strongly orthogonal spanning set. Hence,  $\beta$ is not orthogonal to some root in $S$, which shows the claim.
\end{proof}

Considering Lema \ref{lem:Z0}, Proposition \ref{prop:res to H cap T} and Fact \ref{fact:Kb1}, we obtain  the first statement of Theorem \ref{prop:H=H0} as a corollary.

\begin{cor}
Let $(\pi,V)$ be a holomorphic discrete series and $H_0$ the real subgroup of $G$ associated to the KS-triple as in \eqref{def:Z0}, then $\pi$ restricted to $H_0 \cap T$ is admissible. Therefore, the restriction of $\pi$ to $H_0$ is admissible.
\end{cor}

 Let $\h$ a copy of $\mathfrak{sl}_2(\R)$ as subalgebra of $\g$ whose complexification is generated by the triple  \eqref{eqn:KS-triple}. Henceforth, we assume $\z$ satisfies the hypothesis of Proposition \ref{prop:res to H cap T}.  Then,  any  holomorphic irreducible square integrable representation $(\pi, V)$ restricted to   $\exp(i\mathbb R Z_\h) = H\cap T$ is an admissible representation.
Owing to Fact \ref{fact:res_L}, $\pi$ is an  admissible representation of $H$.   To continue to the proof of Theorem \ref{prop:H=H0}, we show the following.

\begin{prop}   \label{prop:Z=Z0}
Let $Z_\h$ be a characteristic vector  as in \eqref{eqn:KS-triple} and assume that $\beta_\ell (Z_\h) $ is positive. Then, $Z_\h= Z_0$ and $H$ is conjugated  to $H_0$ by an element of  $G$.
\end{prop}

For groups of type $EIII,EVII$ the equality $Z_\h=Z_0$ follows by inspection of tables $X,XIII$ in \cite{Dk1}. We proceed in a proof for all groups $G$, for this we need some other results to start with.

We recall a result of Barbasch-Vogan  (Lemma 3.7.3 in \cite{CM}) on the centralizer $\g_\C^{E_\h}$ of $E_\h$ in $\g_\C.$ Let $\g_\C^\h$ denote the centralizer of $\h$ in $\g_\C$. Thus, $\g_\C^\h$ is a reductive subalgebra of $\g_\C.$ Let $\u$ denote the nilpotent radical of $\g_\C^{E_\h}$, then the result of Barbasch-Vogan stablishes  the direct sum decomposition
  $$
  \g_\C^{E_\h}= \g_\C^\h + \u
  $$
 Since $\h$ is invariant under the Cartan involution associated to the decomposition $\g=\k +\mathfrak p$, we have that    $\g_\C^\h$ is also invariant under the Cartan involution. Moreover, $\g_\C^\h$ is invariant under the conjugation of $\g_\C$ with respect to the real algebra $\g.$ Therefore, we have the decompositions,
$$
 \g_\C^\h = \g_\C^\h \cap \k_\C + \g_\C^\h \cap \mathfrak p_\C \qquad \qquad
\g_\C^\h = (\g_\C^\h \cap \g) + i (\g_\C^\h \cap \g)
 $$
and
 $$
\g_\C^\h \cap \g = \g_\C^\h \cap \k + \g_\C^\h \cap \mathfrak p
$$
In particular, we can conclude the following.

\begin{lem} \label{lem:centralizador h}
Let $\h$ satisfying the conditions of Proposition \ref{prop:res to H cap T}, then  $\g_\C^\h$ is a subalgebra of $\k_\C$.
\end{lem}

\begin{proof}Obviously, we have
$$
\g_\C^{\h}
\subset
\t\oplus\sum_{\underset{\beta(Z_\h)=0}{\pm\beta \in \Psi}} \g_\beta
$$
Under our hypothesis, Proposition \ref{prop:res to H cap T} says that the roots that vanish on $Z_\h$ are  compact. Whence,
$$
\g_\C^{\h} \subset
\t\oplus\sum_{\underset{\beta(Z_\h)=0}{\pm\beta \in \Psi_c}} \g_\beta
\subset \k_\C
$$
Thus, the proof of the lemma is completed.
\end{proof}

We now complete the proof of Proposition \ref{prop:Z=Z0}.  We write $E_\h=\sum_{\gamma \in \Phi_n} c_\gamma E_\gamma$ and
 $\supp(E_\h)=\{ \gamma \in \Phi_n: c_\gamma \not= 0 \}$. We verify that $\supp(E_\h) \subset \Psi_n.$ This follows from, first  $\gamma =\pm (\beta_\ell + \sum_{1 \leq j \leq \ell-1} n_j \beta_j)$ where $n_j \geq 0$ for all $j$, and second $[Z_\h ,E_\h]= 2 E_\h$.

As before, $\mathfrak p^+ =\sum_{\gamma \in \Psi_n} \g_\gamma$. Since, $\Psi$ is holomorphic, we have $\mathfrak p^+$ is invariant under the group $\Ad(K_\C)$.
We now show that the orbit $\Ad(K_\C) E_\h$ is open in $\mathfrak p^+.$ For this, we recall the classification of the orbits of $K_\C$ in $\mathfrak p^+$ due to \cite{RRS}.
Let $S=\{\gamma_1, \dots, \gamma_r \}$ be the Harish-Chandra set of strongly orthogonal roots of $\Psi_n$ given in \eqref{eqn:S}. Then a set of representatives of the $K_\C$-orbits in $\mathfrak p^+$ is
 $$
  E_{\gamma_1} +\dots +E_{\gamma_s} \qquad s=0,1,\dots,r
  $$
There is only one open orbit, which corresponds to $s=r.$ Thus, there exists a $k \in K_\C$ and $1\leq s \leq r$ so that
$$
E_\h = \Ad(k)( E_{\gamma_1} +\dots +E_{\gamma_s})
$$
If it were $s<r,$ then the complex simple algebra $\Ad(k)(\mathfrak{sl}_2(\gamma_r))$ would be contained in the centralizer of $E_\h$ and because of the result of Barbash-Vogan quoted previously, we have  $\Ad(k)(\mathfrak{sl}_2(\gamma_r)) \subset \g_\C^\h.$
 Since, $\gamma_r$ is a noncompact root, and $k \in K_\C$ we would conclude that  $\g_\C^\h \cap \mathfrak p_\C$ has positive dimension, which contradicts Lemma \ref{lem:centralizador h}.
Therefore $s=r$ and the orbit $\Ad(K_\C)E_\h$ is open in $\mathfrak p^+$.
As a consequence we have found an element  $k\in K_\C$ such that
$$
\Ad(k^{-1})E_\h = E_{\gamma_1} +\dots +E_{\gamma_r}=E_0
$$
 Since, the triple $\{\Ad(k)Z_\h, \Ad(k)E_\h=E_0, \Ad(k)F_\h\}$ is normal, applying Theorem 9.4.3 in \cite{CM}, there is $k' \in K_\C$ such that
 $$
 \Ad(k')Z_\h =Z_0 \qquad \text{and}\qquad \Ad(k')E_\h=E_0
 $$
Finally we obtain the KS-triples $(Z_\h , E_\h , F_\h)$, $(Z_0 , E_0 , F_0)$ are  $K_\C$-conjugated. The obvious  map $\psi_0$ from the set of $K_0$-conjugacy classes of KS-triples into the  set of  $K_\C$-conjugacy classes of normal triples is shown to be injective in \cite{Dk1}. Hence,  there exists
$k \in K_0$ which carries $Z_\h$ onto $Z_0$. Since both vectors are $\Psi$-dominant, it follows that they are equal. Also, $k$ carries $H$ onto $H_0$. This ends the proof of Proposition \ref{prop:Z=Z0}.

The proof of the second part of  Theorem \ref{prop:H=H0} follows straightforward from Lemma \ref{lem:Z0} and Proposition \ref{prop:Z=Z0}.

The Hermitian symmetric space $G/K$ is a \textit{tube domain} if it is biholomorphic to a  tube domain. It is well known that $G/K$ is a tube domain if and only if the characteristic vector $Z_0$ defined in \eqref{def:Z0}, is in the center of $\k_\C$.
 
\begin{rmk}  \label{cal: betaz0}
We have $\beta_\ell(Z_0)=2$ and $ \beta_j(Z_0)=0$, for all $1\leq  j \leq \ell-1$,  if and only if $G/K$ is a tube  domain.  Whenever, $G/K$ is not a tube domain,   we have $\beta_\ell(Z_\h)=1$,  and  $\beta(Z_0)=0$ for all the   compact simple roots but one for which we have $\beta(Z_0)=1.$ Indeed, for a holomorphic system it happens  that for any $X$ is the center of $\k$ the value $\beta(X)=\beta_\ell (X)$ for any  $\beta \in \Psi_n.$ Also, by construction,  $\beta_M \in S $ which yields $\beta_M(Z_0)=2$. Thus, if $Z_0 $ belongs to the center of $\k$ we have $\beta(Z_0)=2$ for every root in $\Psi_n,$ which gives $\beta_j (Z_0)=0$ for every compact simple root. Certainly, the hypothesis $\beta(Z_0)=0$ for every compact simple root, together with $\Psi$ holomorphic yields $Z_0$ lies in the center of $\k.$ The hypothesis $\beta_\ell(Z_0)=1$ yields $Z_0$ is not in the center of $\k$ which is equivalent to $G/K$ is not a tube domain. When $\beta_\ell(Z_0)=1,$ since $\beta_M(Z_0)=2$ and that the multiplicity of $\beta_\ell$ in $\beta_M$ is one, we obtain that $\beta_j(Z_0)=1$ for exactly one compact simple root and  the root $\beta_j$ has multiplicity one in the maximal root.
\end{rmk}

\begin{rmk}
 In \cite{Ha}, is shown a necessary condition for a nilpotent orbit to be in the wave front cycles of a tempered representation. More precisely, let $L$ be the Levi subgroup of the centralizer of an element of a nilpotent  orbit in $\g$ and $C(G)$ the center of $G$, the necessary conditions is that $L/C(G)$ is be compact.  Certainly the real nilpotent orbit $\mathcal O_0$ that corresponds via the Kostant-Sekiguchi map to the $K_\C$-orbit $E_0$ is contained in the wave front set of any holomorphic discrete series.  We have observed in a case by case computation for $BDI, EIII, EVII$ that   dimension of $\mathcal O_0$ is minimal among the values of dimension of the real nilpotent orbits that satisfy the condition $L/C(G)$ is compact.
\end{rmk}

\section{Explicit examples} \label{sec:examples}

In this section, for each  Hermitian symmetric pair,   we give the necessary data in order to produce an explicit example of  each KS-triple $\{Z_0, E_0, F_0\}$ as well as the values  $\beta(Z_0)$   for each simple root for the holomorphic system $\Psi$.

In \cite{Oh} is pointed out an explicit realization of the classical real Lie algebras we are dealing with as a subalgebra of a convenient $\mathfrak{su}(a,b)$. These realization have the property that a compactly embedded Cartan subalgebra of the algebras of our interest, consists of the totality of diagonal matrices in the subalgebra. For each example on classical Lie algebras, we point out the algebra $\t_\C$, a holomorphic system $\Psi$, the Harish-Chandra set $S$ as in \ref{eqn:S}, the vector $Z_0$ as in \eqref{def:Z0}, the weights $\beta_j(Z_0), j=0,\dots,\ell $, for all $\beta_j$ in  \eqref{eqn:simple roots}, the weighted Vogan diagram that correspond to the $K_\C$-orbit of $E_0$ (see \cite{Ga}) and the signed Young diagram that corresponds to $E_0$.

From the tables in \cite{Dk1}, we also present on exceptional Lie algebras the Harish-Chandra set $S$ and  the  weighted Vogan diagram associated to the orbit of $E_0$.

\

{\bf AIII, $\mathfrak{su}(p,q)$, $p< q.$}

 In this case
 $$
 \t_\mathbb C=\left\{ D=\diag(h_1, \dots, h_p; k_1, \dots, k_q) : \sum h_j +\sum k_s=0\right\}.
 $$
 We set $\epsilon_j(D)= h_j, \delta_r(D)=k_r.$ Then for a holomorphic system $\Psi$ we choose
 $$
 \Psi_c=\{ \epsilon_r -\epsilon_s, \delta_i -\delta_j, r<s, i<j\} \quad \Psi_n=\{ \epsilon_i -\delta_j, 1 \leq i \leq p, 1 \leq j \leq q \}.
 $$
 The noncompact simple root is $\beta_p =\epsilon_p -\delta_1,$ other simple root we need  is $\beta_q= \delta_{q-p} -\delta_{q-p+1}.$

The Harish-Chandra strongly orthogonal spanning set is
$$S=\{\epsilon_{r} -\delta_{q-r+1}, 1 \leq r \leq p \}.$$
The characteristic vector
 $$
 Z_0 = \diag(1,\dots,1; 0,\dots,0,-1,\dots,-1) \quad \text{where $\pm 1$ repeats $p$  times.}
 $$
The weights $w_j=\beta_j(Z_0)$ are zero for roots other than $\beta_p,
\beta_q=\delta_{q-p} -\delta_{q-p+1}$.  $w_p=\beta_p(Z_0)=1$ and $w_q=\beta_q(Z_0)$ are equal one. Whence, the weighted Vogan diagram for the orbit $K_\C E_0$ is

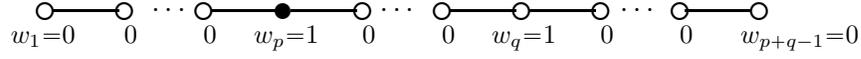
\begin{center}
\setlength{\unitlength}{10pt}
\begin{picture}(15,3)(-6,-2)
  \thicklines
   \put(-12,0){\circle{.6}}
  \put(-9,0){\circle{.6}}
  \put(-6,0){\circle{.6}}
  \put(-3,0){\circle*{.6}}
  \put(0,0){\circle{.6}}
  \put(3,0){\circle{.6}}
  \put(6,0){\circle{.6}}
  \put(9,0){\circle{.6}}
  \put(12,0){\circle{.6}}
  \put(15,0){\circle{.6}}
  \put(-11.7,0){\line(3,0){2.5}}
  \put(-8.7,0){\,\, \dots}
  \put(-5.7,0){\line(3,0){2.5}}
  \put(-2.7,0){\line(3,0){2.5}}
  \put(0.3,0){\,\,\dots}
  \put(3.3,0){\line(3,0){2.5}}
  \put(6.3,0){\line(3,0){2.5}}
  \put(9.2,0){\,\,\,\dots}
  \put(12.2,0){\line(3,0){2.5}}

  \put(-13.3,-1.2){\small{$w_1\tiny{=}0$}}
  \put(-9,-1.2){\small{0}}
  \put(-6,-1.2){\small{0}}
  \put(-4.1,-1.2){\small{$w_p\tiny{=}1$}}
  \put(0,-1.2){\small{0}}
  \put(3,-1.2){\small{0}}
  \put(4.9,-1.2){\small{$w_q\tiny{=}1$}}
  \put(9,-1.2){\small{0}}
  \put(12,-1.2){\small{0}}
  \put(14.3,-1.2){\small{$w_{p+q-1}\tiny{=}0$}}


\end{picture}
\end{center}

 The signed Young diagram for $E_0$ is
$$
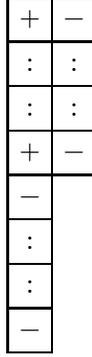
\begin{ytableau}
+ & - \cr
: & : \cr
: & : \cr
+ & - \cr
- \cr
: \cr
: \cr
- \cr
\end{ytableau}
$$
Here, there are $p$ rows   of   length   two and   $q-p$ rows of length one.

   \

{\bf AIII,    $\mathfrak{su}(p,q), p=q$.}

 $$
 \t_\mathbb C=\left\{ D=\diag(h_1, \dots, h_p; k_1, \dots, k_p) : \sum h_j +\sum k_s=0\right\}
 $$
 We set $\epsilon_j(D)= h_j, \delta_r(D)=k_r$, with $1 \leq j,r \leq p$. Then for a holomorphic system $\Psi$ we choose
 $$
 \Psi_c=\{ \epsilon_r -\epsilon_s, \delta_i -\delta_j, r<s, i<j\} \quad
 \Psi_n=\{ \epsilon_i -\delta_j, 1 \leq i,j \leq p \}.
 $$
 The noncompact simple root is $\beta_p =\epsilon_p -\delta_1$.\\
The Harish-Chandra strongly orthogonal spanning set is
$$S=\{\epsilon_{r} -\delta_{q-r+1}, 1 \leq r \leq p \}.$$
The characteristic vector is
$$
Z_0 = \diag(1,\dots,1; -1,\dots,-1) \quad \text{where $\pm 1$ repeats $p$  times}
$$
 Thus, the weights are $w_j=\beta_j(Z_0)=0$  except for $w_p=\beta_p (Z_0)=2 $. So, its weighted Vogan diagram is

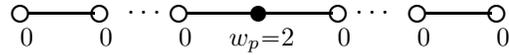
\begin{center}
\setlength{\unitlength}{10pt}
\begin{picture}(9,3)(-6,-2)
  \thicklines
   \put(-12,0){\circle{.6}}
  \put(-9,0){\circle{.6}}
  \put(-6,0){\circle{.6}}
  \put(-3,0){\circle*{.6}}
  \put(0,0){\circle{.6}}
  \put(3,0){\circle{.6}}
  \put(6,0){\circle{.6}}
  \put(-11.7,0){\line(3,0){2.5}}
  \put(-8.7,0){\,\, \dots}
  \put(-5.7,0){\line(3,0){2.5}}
  \put(-2.7,0){\line(3,0){2.5}}
  \put(0.3,0){\,\,\dots}
  \put(3.3,0){\line(3,0){2.5}}

  \put(-12,-1.2){\small{0}}
  \put(-9,-1.2){\small{0}}
  \put(-6,-1.2){\small{0}}
  \put(-4.1,-1.2){\small{$w_p\tiny{=}2$}}
  \put(0,-1.2){\small{0}}
  \put(3,-1.2){\small{0}}
  \put(6,-1.2){\small{0}}


\end{picture}
\end{center}

The signed Young diagram for $E_0$   has  $p$ rows   of   length   two.
$$
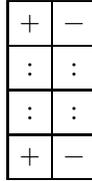
\begin{ytableau}
+ & - \cr
: & : \cr
: & : \cr
+ & - \cr
\end{ytableau}
$$

\bigskip

\bigskip

\bigskip

{\bf BDI,  $\mathfrak{so}(2p+1,2), p\geq 1$.}

The complexification of the toroidal  Cartan subalgebra is
$$
\t_\C =\{ D=\diag(h_1, \dots, h_p, -h_p, \dots, -h_1,0, x_1, -x_1)\}
$$

We set $\epsilon_j(D)=h_j, \delta_1 (D)= x_1$. We fix the holomorphic system of positive roots,
$$
\Psi_c=\{ \epsilon_k, \epsilon_i \pm \epsilon_j, 1 \leq k \leq p, 1 \leq i <j\leq p\}
\quad
\Psi_n =\{ \delta_1, \delta_1 \pm \epsilon_j, 1 \leq j \leq p\}.$$

The noncompact simple root is $\beta_{1} =\delta_1 -\epsilon_1$.

The Harish-Chandra set is $S=\{\delta_1 +\epsilon_1, \delta_1 -\epsilon_1\}$.

 The characteristic vector is $Z_0 = 2 H_{\delta_1} = (0,\dots,0,2,-2)$.

  The weights of the weighted Vogan diagram are zero except the first one,

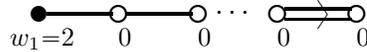
\begin{center}
\setlength{\unitlength}{10pt}
\begin{picture}(9,3)(-6,-2)
  \thicklines
  \put(-6,0){\circle*{.6}}
  \put(-3,0){\circle{.6}}
  \put(0,0){\circle{.6}}
  \put(3,0){\circle{.6}}
  \put(6,0){\circle{.6}}

  \put(-5.7,0){\line(3,0){2.5}}
  \put(-2.7,0){\line(3,0){2.5}}
  \put(0.3,0){\,\,\dots}
  \put(3.3,0.2){\line(3,0){2.5}}
  \put(3.3,-0.2){\line(3,0){2.5}}
  \put(4.5,-0.3){$\rangle$}

  \put(-7.1,-1.2){\small{$w_1\tiny{=}2$}}
  \put(-3,-1.2){\small{0}}
  \put(0,-1.2){\small{0}}
  \put(3,-1.2){\small{0}}
  \put(6,-1.2){\small{0}}


\end{picture}
\end{center}

   The signed Young diagram for $E_0$ has  $ 2p$  rows   of   length   one.

    $$
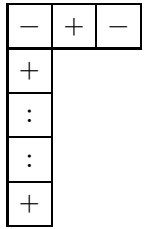
\begin{ytableau}
- & + & - \cr
+    \cr
:    \cr
:    \cr
+   \cr
\end{ytableau}
$$

\bigskip

{\bf BDI, $\mathfrak{so}(2p,2), p\geq 2.$}

This case is similar the previous one. The  Cartan subalgebra is
$$
\t_\C =\{ D=\diag(h_1, \dots, h_p, -h_p, \dots, -h_1,0, x_1, -x_1)\}
$$

We set $\epsilon_j(D)=h_j$, $\delta_1 (D)= x_1$.
The   the holomorphic system we consider is
$$
\Psi_c=\{  \epsilon_i \pm \epsilon_j,  1 \leq i <j\leq p\} \quad
\Psi_n =\{ \ \delta_1 \pm \epsilon_j, 1 \leq j \leq p\}
$$

The noncompact simple root is $\beta_{p} =\delta_1 -\epsilon_1$.

The Harish-Chandra set is $S=\{\delta_1 +\epsilon_1, \delta_1 -\epsilon_1\}$.

The characteristic vector is $Z_0 = 2 H_{\delta_1} = (0,\dots,0,2,-2)$.

  The weights of the weighted Vogan diagram are zero except the first one.

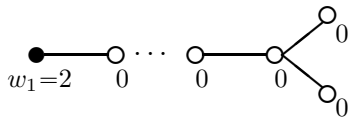
\begin{center}
\setlength{\unitlength}{10pt}
\begin{picture}(11,9)(-6,-5)
  \thicklines
  \put(-6,0){\circle*{.6}}
  \put(-3,0){\circle{.6}}
  \put(0,0){\circle{.6}}
  \put(3,0){\circle{.6}}
  \put(5,1.5){\circle{.6}}
  \put(5,-1.5){\circle{.6}}
  \put(-5.7,0){\line(3,0){2.4}}
  \put(-2.7,0){\,\,\dots}
  \put(0.3,0){\line(3,0){2.4}}
  \put(3.3,0){\line(5,4){1.5}}
  \put(3.3,0){\line(5,-4){1.5}}

 \put(-7.1,-1.2){\small{$w_1\tiny{=}2$}}
  \put(-3,-1.2){\small{0}}
  \put(0,-1.2){\small{0}}
  \put(3,-1.2){\small{0}}
  \put(5.3,0.5){\small{0}}
  \put(5.3,-2.3){\small{0}}

\end{picture}
\end{center}

   The signed Young diagram for $E_0$ has  $ 2p-1$  rows   of   length   one.

  $$
\begin{ytableau}
- & + & - \cr
+  \cr
:   \cr
:  \cr
+  \cr
\end{ytableau}
$$

\bigskip

{\bf CI, $\mathfrak{sp}(n,\mathbb R).$}

The complex Cartan subalgebra is
$$
\t_\C =\{ D=\diag(h_1, \dots, h_n, -h_n, \dots, -h_1) \}
$$

We set $\epsilon_j(D)=h_j$, $1\leq j \leq n$. The holomorphic system we consider is
$$
\Psi _c =\{ \epsilon_i -\epsilon_j, 1 \leq i < j\leq n \} \quad
\Psi_n=\{ \epsilon_k + \epsilon_r, 1 \leq k \leq r \leq n \}
$$

The noncompact simple root is $\beta_n = 2\epsilon_n$.

The set $S=\{ 2\epsilon_1, \dots, 2\epsilon_n \}$.

 The characteristic vector is $Z_0 =(1, \dots, 1, -1,  \dots, -1)$.

The weights of the weighted Vogan diagram are zero except the last one, $w_n=\beta_n(Z_0)=2$.

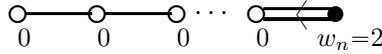
\begin{center}
\setlength{\unitlength}{10pt}
\begin{picture}(9,3)(-6,-2)
  \thicklines
  \put(-6,0){\circle{.6}}
  \put(-3,0){\circle{.6}}
  \put(0,0){\circle{.6}}
  \put(3,0){\circle{.6}}
  \put(6,0){\circle*{.6}}

  \put(-5.7,0){\line(3,0){2.5}}
  \put(-2.7,0){\line(3,0){2.5}}
  \put(0.3,0){\,\,\dots}
  \put(3.3,0.2){\line(3,0){2.5}}
  \put(3.3,-0.2){\line(3,0){2.5}}
  \put(4.5,-0.3){$\langle$}

  \put(-6,-1.2){\small{0}}
  \put(-3,-1.2){\small{0}}
  \put(0,-1.2){\small{0}}
  \put(3,-1.2){\small{0}}
  \put(5.3,-1.2){\small{$w_n\tiny{=}2$}}


\end{picture}
\end{center}

The signed Young diagram for $E_0$ has $ n $ rows   of   length   two.

 $$
\begin{ytableau}
+ & -  \cr
: &  : \cr
: &  :  \cr
+ &  - \cr
\end{ytableau}
$$

\bigskip

{\bf DIII,  $\mathfrak{so}^\star(2p)$, $p=2k$.}

The complex Cartan subalgebra is
 $$
 \t_\C =\{ D=\diag(h_1, \dots, h_p, -h_p, \dots, -h_1) \}
 $$

We set $\epsilon_j(D)=h_j$, $1\leq j \leq p$. The holomorphic system we consider is
$$
\Psi_c =\{ \epsilon_i -\epsilon_j, 1 \leq i < j\leq p \} \quad
\Psi_n=\{ \epsilon_s + \epsilon_r, 1 \leq s < r \leq p \}
$$

The noncompact simple root is  $\beta_p = \epsilon_{p-1}+ \epsilon_p$.

The Harish-Chandra set is $S=\{ \epsilon_1 +\epsilon_2, \epsilon_3 +\epsilon_4,   \dots, \epsilon_{2k-1} +\epsilon_{2k} \}$.

Characteristic vector is
$$
Z_0 = \sum_{1\leq j \leq k} H_{\epsilon_{2j-1} +\epsilon_{2j}} = (1, \dots, 1, -1,  \dots, -1)
$$

The weights $w_j=\beta_j(Z_0)$ are $w_p=2, w_j=0$ for $j\not=p$. So the weighted Vogan diagram is the following.

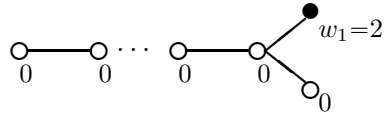
\begin{center}
\setlength{\unitlength}{10pt}
\begin{picture}(11,5)(-6,-3)
  \thicklines
  \put(-6,0){\circle{.6}}
  \put(-3,0){\circle{.6}}
  \put(0,0){\circle{.6}}
  \put(3,0){\circle{.6}}
  \put(5,1.5){\circle*{.6}}
  \put(5,-1.5){\circle{.6}}
  \put(-5.7,0){\line(3,0){2.4}}
  \put(-2.7,0){\,\,\dots}
  \put(0.3,0){\line(3,0){2.4}}
  \put(3.3,0){\line(5,4){1.5}}
  \put(3.3,0){\line(5,-4){1.5}}

 \put(-6,-1.2){\small{0}}
  \put(-3,-1.2){\small{0}}
  \put(0,-1.2){\small{0}}
  \put(3,-1.2){\small{0}}
  \put(5.3,0.5){\small{$w_1\tiny{=}2$}}
  \put(5.3,-2.3){\small{0}}
\end{picture}
\end{center}

The signed Young diagram for $E_0$ has $2k$ rows of  length   two.

  $$
\begin{ytableau}
+ & - \cr
: & : \cr
: & : \cr
+ & - \cr
\end{ytableau}
$$

\bigskip

{\bf DIII,  $\mathfrak{so}^\star(2p)$, $p=2k+1$.}

This case is similar to the previous one. The difference is that the characteristic vector is
$$
Z_0 =(1, \dots, 1, -1,  \dots, -1,0)
$$
Thus, the weights $w_{p-1}, w_p$ are equal to $(\epsilon_{p-1} \pm \epsilon_p)(Z_0)= 1$ and the others are zero.

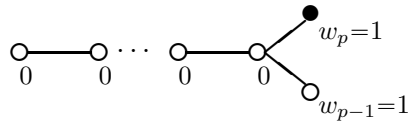
\begin{center}
\setlength{\unitlength}{10pt}
\begin{picture}(11,8)(-6,-5)
  \thicklines
  \put(-6,0){\circle{.6}}
  \put(-3,0){\circle{.6}}
  \put(0,0){\circle{.6}}
  \put(3,0){\circle{.6}}
  \put(5,1.5){\circle*{.6}}
  \put(5,-1.5){\circle{.6}}
  \put(-5.7,0){\line(3,0){2.4}}
  \put(-2.7,0){\,\,\dots}
  \put(0.3,0){\line(3,0){2.4}}
  \put(3.3,0){\line(5,4){1.5}}
  \put(3.3,0){\line(5,-4){1.5}}

 \put(-6,-1.2){\small{0}}
  \put(-3,-1.2){\small{0}}
  \put(0,-1.2){\small{0}}
  \put(3,-1.2){\small{0}}
  \put(5.3,0.5){\small{$w_p\tiny{=}1$}}
  \put(5.3,-2.3){\small{$w_{p-1}\tiny{=}1$}}
\end{picture}
\end{center}
The signed Young diagram for $E_0$ has  $2k$ rows of  length   two.

 $$
\begin{ytableau}
+ & - \cr
: & : \cr
: & : \cr
+ & - \cr
+ \cr
- \cr
\end{ytableau}
$$

{\bf EIII, $ \mathfrak e_{6 (-14)} $. }

We follow the notation for the simple roots as set for Bourbaki. We fix as non compact simple root   $\beta_6$. The Harish-Chandra set in this case is  $S=\{ \gamma_1 =122321, \gamma_2 =101111\} $.

From table X of \cite{Dk1}, we extract that there is only  one characteristic vector $Z_\h$ so that $\beta_\ell(Z_h)>0$, and we obtain $Z_\h=Z_0$   as in \eqref{def:Z0}.  The  weighted Vogan diagram for the nilpotent orbit determinate by $E_0$ is

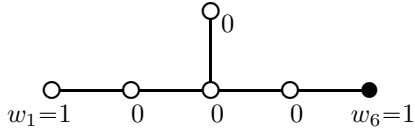
\begin{center}
\setlength{\unitlength}{10pt}
\begin{picture}(7,5)(-5,-1)
  \thicklines
  \put(-6,0){\circle{.6}}
  \put(-3,0){\circle{.6}}
  \put(0,0){\circle{.6}}
  \put(0,3){\circle{.6}}
  \put(3,0){\circle{.6}}
  \put(6,0){\circle*{.6}}
  \put(-5.7,0){\line(3,0){2.4}}
  \put(-2.7,0){\line(3,0){2.4}}
  \put(0.3,0){\line(3,0){2.4}}
  \put(0,.3){\line(0,3){2.4}}
  \put(3.3,0){\line(3,0){2.4}}

\put(-7.7,-1.2){\small{$w_1\tiny{=}1$}}
  \put(-3,-1.2){\small{0}}
  \put(0,-1.2){\small{0}}
   \put(0.4,2.2){\small{0}}
  \put(3,-1.2){\small{0}}
  \put(5.3,-1.2){\small{$w_6\tiny{=}1$}}

\end{picture}
\end{center}

\bigskip

{\bf EVII, $\mathfrak e_{7(-25)}.$}

We fix the holomorphic root system such that the noncompact simple root is  $\alpha_7$. The Harish-Chandra set is
 $$
 S=\{\gamma_1 = 22343221, \gamma_2 =01122221, \gamma_3 =\beta_7=00000001 \}
 $$

From table XIII of \cite{Dk1}, we read that  the only $K_\C$-orbit in $\p_\C$ with characteristic vector $Z_\h$ so that $\beta_\ell(Z_\h)>0 $ is for $Z_\h=Z_0$  as in \ref{def:Z0}. The weighted Vogan diagram is

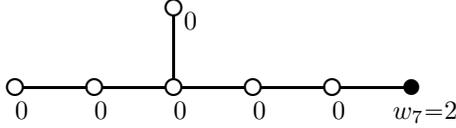
\begin{center}
\setlength{\unitlength}{10pt}
\begin{picture}(14,5)(-6,-2)
  \thicklines
  \put(-6,0){\circle{.6}}
  \put(-3,0){\circle{.6}}
  \put(0,0){\circle{.6}}
  \put(0,3){\circle{.6}}
  \put(3,0){\circle{.6}}
  \put(6,0){\circle{.6}}
  \put(9,0){\circle*{.6}}
  \put(-5.7,0){\line(3,0){2.4}}
  \put(-2.7,0){\line(3,0){2.4}}
  \put(0.3,0){\line(3,0){2.4}}
  \put(0,.3){\line(0,3){2.4}}
  \put(3.3,0){\line(3,0){2.4}}
  \put(6.3,0){\line(3,0){2.5}}

\put(-6,-1.2){\small{0}}
  \put(-3,-1.2){\small{0}}
  \put(0,-1.2){\small{0}}
   \put(0.4,2.2){\small{0}}
  \put(3,-1.2){\small{0}}
  \put(6,-1.2){\small{0}}
  \put(8.3,-1.2){\small{$w_7\tiny{=}2$}}

\end{picture}
\end{center}

 \

\section{Multiplicities}\label{sec:multiplicities}

In this section we apply the formula for multiplicities obtained in \cite{DV} to the particular case of the pair $(G,H_0).$ Henceforth, $\Psi$ denotes a holomorphic system of positive roots for $\Phi (\g, \t).$  To begin with we recall the necessary notation to state the results. In the notation of \cite{DV}  the pair $(G,H)$ for our case is $(G,H_0)$. We have $T \subset K \subset G$ as before and $U:=L:=H_0 \cap K =H_0 \cap T=\exp(\R i Z_0)$,    $\u =\R i Z_0$, we denote by  $\mathfrak z_\k =$ the center of $\k$. We define $\varphi \in \u^*$   by $\varphi (Z_0)=1$. Thus $\Phi(\h_0, \u)=\{ \pm 2 \varphi \}.$  Let $\k_\mathfrak z=\k^{Z_0}$ denote the centralizer of $Z_0$ in $\k$ and $\Phi_\mathfrak z$ the root system for $(\k_\mathfrak z, \t).$ Thus,
$$
\Phi_\mathfrak z =\{ \alpha \in \Phi(\k, \t) : \alpha (\u)=0 \}
$$
 By Remark \ref{cal: betaz0}, if the Hermitian symmetric space $G/K$ is a tube domain, then $Z_0 \in \mathfrak z_\k$, $\Phi_\mathfrak z=\Phi(\k,\t)$ and the analytic subgroup of $G$ with Lie algebra $\k_\mathfrak z$ is $K_\mathfrak z =K$.
If $G/K$ is not a tube domain, $Z_0 \notin \mathfrak z_\k$, then owing to $\beta_j(Z_0)=0$ for all compact simple roots but one, the semisimple factor of $\k_\mathfrak z$  has rank $\ell -2.$ The list of the triples $(\g, \k,\k_\mathfrak z)$ that not correspond to tube domains is:

\bigskip
\begin{center}
  \begin{tabular}{|c|c|c|c|}
            \hline
            $\g$ &  $\mathfrak{su}(p,q),\, p<q$ &  $\mathfrak{so}^*(2(2k+1))$  & $ \mathfrak e_{6(-14)} $ \\
            \hline
  $\k$ & $\mathfrak{su}(p)+\mathfrak{su}(q)+\mathfrak z_\k $  & $ \mathfrak{su}(2k+1)+\mathfrak z_\k$ & $\mathfrak{so}(10)+\mathfrak z_\k $  \\
  \hline
    $\k_\mathfrak z $ & $\mathfrak{su}(p)+\mathfrak{su}(q-p)+\mathfrak{su}(p)+\t$ & $ \mathfrak{su}(2k)+\t$  & $\mathfrak{so}(8)+\t$ \\
              \hline
          \end{tabular}
          \end{center}

\bigskip

The set of equivalence classes of  irreducible square integrable representations of $G$ is parameterized by the set of  Harish-Chandra parameters. These parameters are $\lambda \in i\t^*$ so that $\lambda +\rho$ lifts to a character of $T$, here $\rho$ is equal to one half of the sum of the elements in $\Psi,$
\begin{equation} \label{eqn:lambda dominant}
(\lambda, \alpha)>0 \,\,\mbox{for all} \,\,  \alpha \in \Psi_c \quad \text{ and }\quad
 (\lambda, \alpha) \not= 0 \,\, \mbox{ for all} \,\,  \alpha \in \Phi_n.
\end{equation}

 The   holomorphic irreducible square integrable representation corresponds to parameters which also satisfy $(\lambda,\alpha)>0 \,\, \mbox{for all} \,\, \alpha \in \Psi_n.$

The set of  Harish-Chandra parameters which corresponds to the irreducible square integrable representations of $H_0$ is
$P :=\{ n \varphi : n \in \Z \smallsetminus \{0\}\,\, \}$.

The set of Harish-Chandra parameters for a compact Lie group $L$ is equal to the set of strictly dominant integral weights for $L$, equivalently, the set of Harish-Chandra parameters is equal to the set of  infinitesimal character of the set of irreducible representations of $L$. We denote the parametrization by $\mu \leftrightarrow (\pi_\mu^L,  V_\mu^L ) .$

Let $(\pi_\lambda, V_\lambda^G )$ be a holomorphic irreducible square integrable representation. Therefore, the restriction $\res_{H_0} (\pi_\lambda)$ of $\pi_\lambda$ to the subgroup $H_0$ is an admissible discretely decomposable representation of $H_0$.
For $\mu \in P,$ let $(\sigma_\mu, V_\mu^{H_0})$ denote the irreducible square integrable representation of $H_0$ of Harish-Chandra parameter $\mu$.
Let $m(\pi_\lambda, \sigma_\mu)=\Hom_{H_0}(\sigma_\mu, \pi_\lambda)$ denote the multiplicity of $\sigma_\mu$ in $\res_{H_0} (\pi_\lambda).$ Therefore, we have a Hilbert discrete sum decomposition
$$
\res_{H_0} (\pi_\lambda) = \sum_{\mu \in P} m(\pi_\lambda, \sigma_\mu) V_\mu^{H_0}
$$
We write the restriction to $K_\mathfrak z$ of the  lowest $K$-type $\pi_{\lambda +\rho_n}^K$ for $(\pi_\lambda, V_\lambda^{G})$ as  $$
\res_{K_\mathfrak z} (\pi_{\lambda +\rho_n}^K)= \sum_{1 \leq j \leq s} m(\pi_{\lambda +\rho_n}^K, \pi_{\mu_j +\rho_\mathfrak z}^{K_\mathfrak z}) V_{\mu_j +\rho_\mathfrak z}^{K_\mathfrak z}
$$
where $\rho_n$ is the half sum of the roots in $\Psi_n$ and $\rho_\mathfrak z$ is the half sum of the roots in $\Phi_\mathfrak z \cap \Psi$.

\begin{thm} \label{thm:multiplicity}
\
\begin{enumerate}
\item[(i)] $m(\pi_\lambda, \sigma_\mu)>0$ if and only if $\mu$ belongs to the set
$$
\{[(\mu_j +\rho_\mathfrak z )(Z_0) +n-1]\varphi : 1 \leq j \leq s,\, n\geq 0 \}
$$

\item[(ii)] The multiplicity $m(\pi_\lambda, \sigma_{m\varphi})$ is equal to
$$
\sum_{\underset { \mu_j(Z_0) + n =m}{j ,n}  }  m(\pi_{\lambda +\rho_n}^K, \pi_{\mu_j +\rho_\mathfrak z}^{K_\mathfrak z}) \sum_{h=0}^{[\frac{n}{2}]}\binom{n-2h+c-1}{c-1}  \binom{h+d-1}{d-1}.
$$
\end{enumerate}
Here $c= \vert \{ \gamma \in  \Psi_n : \gamma(Z_0)=1 \} \vert$ and $ d+1= \vert \{ \gamma \in  \Psi_n : \gamma(Z_0)=2 \} \vert.$
\end{thm}

\begin{rmk} By work of M. Vergne, Jacobsen, Oshima and Mollers, it follows that $res_{H_0}(\pi_\lambda)$ is equal to a discrete Hilbert sum  of holomorphic representations, our contribution is to determinate the Harish-Chandra  parameters for $H_0$ that contributes and the respective multiplicities. However, we obtain an independent proof of the result.
\end{rmk}

The proof of the theorem will take up the rest of this section. It requires more notation. To start,
we consider the restriction $\q_\u : \t^* \to \u^*$ and the multiset
$\Delta(\k/\l,\u):=q_\u( \Psi(\k,\t) \smallsetminus \Phi_{\mathfrak z})$.
 So, we have
\begin{equation} \label{eqn:Delta u}
\Delta(\k/\l,\u)= \left\{ \begin{array}{ll} \emptyset & \mbox{if $G/K$ is a tube domain} \\
\{ \underbrace{\,\varphi,\dots,\varphi}_{a} \} & \mbox{if $G/K$ is not a tube domain}
\end{array}
\right.
\end{equation}
Here $a=\vert \{ \alpha \in \Psi_c : q_\u(\alpha)(Z_0)=1 \}=\frac12 \dim(K/K_\mathfrak z)$. In fact, when $G/K$ is a tube domain then $ iZ_0 \in \mathfrak z_\k$ and  the first claim is obvious.  For a not tube domain Remark \ref{cal: betaz0} yields  that $\alpha(Z_0)=1 $ or $\alpha(Z_0)=0$ for  $\alpha \in \Psi_c$.
For $w$ in the Weyl group  $W_K$ of $K$, we compute the multiset $$S_w^{H_0}:=[q_\u (w\Psi_n)\cup \Delta(\k/\l,\u)]\smallsetminus \Phi(\h_0, \u).$$ Since, $\Psi$ is holomorphic and $w \in W_K$ it follows that  $w\Psi_n=\Psi_n.$ Hence, $S_w^{H_0}$ does not depend on $w.$ We have,
\begin{equation} \label{eqn:q_u}
q_\u(\Psi_n)=\{ \underbrace{ \varphi, \dots, \varphi }_{c}, \underbrace{ 2\varphi, \dots, 2\varphi }_{d+1} \}.
\end{equation}


If  $G/K$ is a tube domain, then $c=0$ and $d+1=\vert \Psi_n \vert$. Indeed,  in Remark \ref{cal: betaz0} we show  $\gamma (Z_0)=2$ for all noncompact positive root. Therefore, from \eqref{eqn:Delta u}, \eqref{eqn:q_u} and the previous computation,  for a tube domain we obtain,
\begin{equation} \label{eqn:D}
S_w^{H_0}=\{ \underbrace{ 2\varphi, \dots, 2\varphi}_{d}\}.
\end{equation}

If $G/K$ is not a tube domain the values of $c$ and $d$ are in the following table.

\smallskip
\begin{center}
\begin{tabular}{|c|c|c|c|}
            \hline
            $\g$ &  $\mathfrak{su}(p,q),\, p<q$ &  $\mathfrak{so}^*(2(2k+1))$  & $ \mathfrak e_{6(-14)} $ \\
            \hline
  $c$ & $(q-p)p$ & $2k$ & $8 $  \\
                \hline
  $d+1$ & $p^2$ & $k(2k-1)$ & $8 $  \\
                \hline
          \end{tabular}
\end{center}
\smallskip

Hence,
\begin{equation} \label{eqn:H}
S_w^{H_0}=\{\underbrace{\varphi,\dots,\varphi}_{c},\underbrace{2\varphi,\dots,2\varphi}_{d}  \} \cup \Delta(\k/\l,\u).
\end{equation}

For $\nu \in i\t^*$ (resp. $\nu \in i\u^*$), $\delta_\nu$ denotes the Dirac distribution on $i\t^*$ (resp. on $i\u^*$) defined by $\nu$.  Let $(q_\u)_* (\delta_\nu)$ be the push-forward of $\delta_\nu$ from $i\t^*$ to $i\u^*.$ Thus, $(q_\u)_* (\delta_\nu)=\delta_\nu $. Let
$$
y_\nu =\sum_{n=0}^\infty \delta_{n\nu +\frac{\nu}{2}}, \qquad z_\nu = \sum_{n=0}^\infty \delta_{n\nu }.
$$
 For a strict finite set $T =\{\nu_1, \dots, \nu_t \} \subset i \t^*$ we define
 $$
 y_T =y_{\nu_1}* \dots * y_{\nu_t}=\underset {\nu \in T} \Asterisk \, y_\nu.
 $$
  Here $*$ means convolution of distributions.  Let $\varpi_\mathfrak z( \lambda):= \prod_{\alpha \in \Psi \cap \Phi_\mathfrak z}  \frac{ \lambda(\alpha)}{\rho_\mathfrak z (\alpha)}$    Then, in \cite{DV} is shown the following equality of distributions on $i\u^* $,
\begin{equation} \label{eqn:K}
 \sum_{\mu \in P} m(\pi_\lambda, \sigma_\mu) \delta_\mu = \sum_{w \in W_\mathfrak z \backslash W_K} \epsilon(w) \varpi_\mathfrak z( w\lambda) \delta_{q_\u(w\lambda)} * y_{S_w^{H_0}}.
\end{equation}
where $\epsilon (w)$ is the sign of $w$.
The validity of the above equality follows by Lemma \ref{lem:Z0} in \cite{DV} because  Condition (C) is satisfied.
We now show Theorem \ref{thm:multiplicity} for $G$ so that $G/K$ is a tube  domain.  For a holomorphic system $\Psi$  we always have  the equality
\begin{equation}\label{eqn:varpi}
\varpi_\mathfrak z( w\lambda)= \varpi_\mathfrak z( w(\lambda+\rho_n)).
\end{equation}

 Now $\k_\mathfrak z=\k$, hence we have
 $\varpi_\mathfrak z(  \lambda)= \dim V_{\lambda+\rho_n}^K$ which is equal to the dimension of lowest $K$-type of $\pi_\lambda.$       Then, by \eqref{eqn:K}, \eqref{eqn:D} together  and the above consideration gives
 $$
  \sum_{\mu \in P} m(\pi_\lambda, \sigma_\mu) \delta_{\mu}=
  \dim V_{\lambda +\rho_n }^K \, \delta_{\lambda (Z_0)\varphi} *y_{2\varphi}^{\vert \Psi_n \vert -1}.
  $$

 Obviously, $y_\nu =\delta_{\frac{\nu}{2}} * z_\nu, $  $q_\u(\rho_n)=\vert \Psi_n \vert \varphi$ and
 $$
 y_{2\varphi}^{\vert \Psi_n\vert -1}= z_{2\varphi}^{\vert \Psi_n\vert -1} * \delta_{-\varphi}* \delta_{q_\u(\rho_n)}.
 $$

For $r,s$ positive integers, it readily follows that
\begin{equation} \label{eqn:M}
z_{ r \, \varphi}^{s}=\sum_{t=0}^\infty \binom{t+s-1}{s-1} \delta_{ t \, r\,\varphi }.
\end{equation}
Hence,  we obtain
$$
\sum_{\mu \in P} m(\pi_\lambda, \sigma_\mu) \delta_{\mu} =
\sum_{t \geq 0}  \dim V_{\lambda +\rho_n }^K \,\binom{t +\vert \Psi_n\vert-2}{\vert \Psi_n \vert -2} \,  \delta_{ [ (\lambda +\rho_n)(Z_0)+  2 t -1] \varphi }.
$$
Therefore, whenever $G/K$ is a tube domain,  the Harish-Chandra parameters that contribute to $\res_{H_0}(\pi_\lambda)$ are $[(\lambda +\rho_n)(Z_0)+ 2t -1]\varphi =[\lambda(Z_0)+d+2t]\varphi$, $t=0,1,\dots$, and the respective multiplicities   are exactly the numbers
$\binom{t +\vert \Psi_n\vert-2}{\vert \Psi_n \vert -2} \dim V_{\lambda +\rho_n }^K  =  \,\binom{t +d-1}{d -1} \dim V_{\lambda +\rho_n }^K$.

To follow, we show the multiplicity formula when $G/K$ is not a tube domain. Hence,  $K_\mathfrak z$ is a proper subgroup of $K$ and $iZ_0$ is not in the center of $\k.$ We manipulate on the right hand side of  formulae \eqref{eqn:K}. Since $\Psi$ is a holomorphic system, $w\rho_n =\rho_n \,\, \mbox{for} \, w \in W_K$. Hence, $q_\u(\rho_n)= [\frac{c}{2}+d+1]\varphi$ and
$ y_{q_\u(\Psi_n)\smallsetminus \Phi(\h_0,\u)}= z_{q_\u(\Psi_n)\smallsetminus \Phi(\h_0,\u)} * \delta_{[\frac{c}{2}+d]\varphi}$.
Hence,   the right hand side of \eqref{eqn:K} becomes equal to

\begin{equation}\label{eqn:J}
 \sum_{w \in W_\mathfrak z \backslash W_K}
\epsilon(w) \varpi_\mathfrak z( w(\lambda +\rho_n)) \delta_{q_\u(w(\lambda+\rho_n))}*  \underset{\gamma \in q_\u(\Psi\smallsetminus \Phi_\mathfrak z)}{\Asterisk} y_\gamma \,*
\underset{\beta \in q_\u(\Psi_n)\smallsetminus \Phi(\h_0)} {\Asterisk}z_{\beta}\,* \delta_{-\varphi}.
\end{equation}

 In the language of discrete Heaviside distributions, the restriction of the lowest $K$-type $\pi_{\lambda+\rho_n}^K$ of $\pi_\lambda$ to $U$  is represented by
\begin{equation}\label{eqn:I}
  \sum_{w \in W_\mathfrak z \smallsetminus W_K} \epsilon(w) \varpi_\mathfrak z( w(\lambda +\rho_n)) \delta_{q_\u(w(\lambda+\rho_n))} \,* \underset {\gamma \in q_\u(\Psi_c\smallsetminus \Phi_\mathfrak z)}{\Asterisk} y_\gamma
\end{equation}
  The restriction of $\pi_{\lambda+\rho_n}^K$   to $U$ can be represented as the restriction of $\pi_{\lambda+\rho_n}^K$   to $K_\mathfrak z$ and then we decompose the resulting representation of $K_\mathfrak z$ as $U$-module. Let $\mu_1+\rho_\mathfrak z, \dots, \mu_s+\rho_\mathfrak z,$ denote the infinitesimal characters  for the irreducible constituents of $\res_{K_\mathfrak z}(\pi_{\lambda+\rho_n}^K$). Here we take $\mu_j$ dominant with respect to $\Psi \cap \Phi_\mathfrak z.$ Then, we have the equality
  $$ \res_U(\pi_{\lambda +\rho_n}^K)=\sum_{1\leq j\leq s} m(\pi_{\lambda +\rho_n}^K, \pi_{\mu_j +\rho_\mathfrak z}^{K_\mathfrak z}) V_{\mu_j +\rho_\mathfrak z}^{U}.
  $$
  Therefore, we have that \eqref{eqn:I} is equal to
\begin{equation}\label{eqn:N}
  \sum_{1\leq j\leq s} m(\pi_{\lambda +\rho_n}^K, \pi_{\mu_j +\rho_\mathfrak z}^{K_\mathfrak z})   \delta_{(\mu_j +\rho_\mathfrak z  )(Z_0) \varphi}.
\end{equation}
 Putting together the new expression for \eqref{eqn:I} and \eqref{eqn:J}, we obtain that the right hand side of \eqref{eqn:K} is equal to
 $$
 \sum_{1\leq j\leq s} m(\pi_{\lambda +\rho_n}^K, \pi_{\mu_j +\rho_\mathfrak z}^{K_\mathfrak z})  \delta_{(\mu_j +\rho_\mathfrak z )(Z_0) \varphi} \,*
 \underset { \gamma \in q_\u(\Psi_n)\smallsetminus \Phi(\h_0,\u)} \Asterisk z_\gamma \,* \delta_{\varphi}.$$
After we recall \eqref{eqn:H} and we apply \eqref{eqn:N} to the previous formula,   we obtain

\begin{equation}
\begin{aligned}
&\qquad\quad \sum_{\mu \in P} m(\pi_\lambda, \sigma_\mu) \delta_{\mu} = \\
 &\sum_{ t\geq 0, h \geq 0} [ \sum_{j} m(\pi_{\lambda +\rho_n}^K, \pi_{\mu_j +\rho_\mathfrak z}^{K_\mathfrak z})\binom{t+c-1}{c-1} \binom{h+d-1}{d-1}  \delta_{[(\mu_j +\rho_\mathfrak z)(Z_0)+ t+ 2h-1] \varphi }.
 \end{aligned}
\end{equation}

Hence,
\begin{equation}
\begin{aligned}
&\qquad\quad \sum_{\mu \in P} m(\pi_\lambda, \sigma_\mu) \delta_{\mu}=\\
&\sum_{j=1}^s     m(\pi_{\lambda +\rho_n}^K, \pi_{\mu_j +\rho_\mathfrak z}^{K_\mathfrak z})\sum_{ n\geq 0}  \sum_{h=0}^{[\frac{n}{2}]}\binom{n-2h+c-1}{c-1} \binom{h+d-1}{d-1}  \delta_{[(\mu_j +\rho_\mathfrak z)(Z_0)+ n-1] \varphi }.
 \end{aligned}
\end{equation}

Therefore a  Harish-Chandra parameter $m\varphi$   of an irreducible $H_0$-factors for $\res_{H_0}(\pi_\lambda)$ belongs to the set
$$
\{ [(\mu_j +\rho_\mathfrak z)(Z_0) + n-1] \varphi : n=0,1 \dots, j=1,\dots,s \},
$$
and the respective multiplicity is
$$
 \sum_{\underset {\mu_j(Z_0) + n =m}  {j ,n}} m(\pi_{\lambda +\rho_n}^K, \pi_{\mu_j +\rho_\mathfrak z}^{K_\mathfrak z}) \sum_{h=0}^{[\frac{n}{2}]}\binom{n-2h+c-1}{c-1} \binom{h+d-1}{d-1}.
 $$
 Now, the proof of Theorem \ref{thm:multiplicity} has been completed.

\begin{rmk}
In the above formulae for either Harish-Chandra parameters or multiplicities, if we make $c$ equal to zero, we obtain the formula for the tube type case. This is the reason we have left the summand $\rho_\mathfrak z(Z_0)$ even thought it is equal to zero when $G/K$ is a tube type domain.
\end{rmk}

\begin{rmk}
The decomposition of the adjoint representation of $\g_\C$ restricted to $\h_0$ is,
\begin{enumerate}
\item[(i)]
 When $G/K$ is a tube domain,
$$
\g_\C = \bigoplus_{1}^{d+1}\, ( \mathbb C^3, 2\varphi)\, \oplus \bigoplus_1^{\dim \k -d-1}  (\C, 0\varphi)
$$
\item[(ii)]  When $G/K$ is not a tube domain,
$$
\g_\C = \bigoplus_{1}^{d+1}\, ( \C^3, 2\varphi) \,\oplus\, \bigoplus_1^c (\C^2, \varphi) \,\oplus \bigoplus_1^{\dim \k -d-1}  (\C, 0\varphi)
$$
\end{enumerate}
Whence, the coefficients $c$ and $d+1$ represent the  multiplicity of the distinct  irreducible constituents  of the $\h_0$-module $ \g_\C$.
 \end{rmk}

\section{ Acknowledgments}

 We would like to thank Michel Duflo for comments at an early stage of this note.

\providecommand{\MR}{\relax\ifhmode\unskip\space\fi MR }
\providecommand{\MRhref}[2]{%
  \href{http://www.ams.org/mathscinet-getitem?mr=#1}{#2}
}
\providecommand{\href}[2]{#2}

\end{document}